\documentclass[11pt]{amsart}
\usepackage{prettyref}
\usepackage{amsthm}
\usepackage{amsmath}
\usepackage{amssymb}

\makeatletter
\theoremstyle{plain}
\newtheorem{thm}{Theorem}
  \theoremstyle{plain}
  
  \theoremstyle{definition}
  \newtheorem{defn}[thm]{Definition}
  \theoremstyle{plain}
  \newtheorem{lem}[thm]{Lemma}
  \newcounter{casectr}
  \newenvironment{caseenv}
  {\begin{list}{{\itshape\ Case} \arabic{casectr}.}{%
   \setlength{\leftmargin}{\labelwidth}
   \addtolength{\leftmargin}{\parskip}
   \setlength{\itemindent}{\listparindent}
   \setlength{\itemsep}{\medskipamount}
   \setlength{\topsep}{\itemsep}}
   \setcounter{casectr}{0}
   \usecounter{casectr}}
  {\end{list}}
  \theoremstyle{plain}
  \newtheorem{cor}[thm]{Corollary}
  \theoremstyle{plain}
  \newtheorem*{cor*}{Corollary}
  \theoremstyle{plain}
  \newtheorem{question}[thm]{Question}

\makeatother

\begin{document}

\def\COMMENT#1{}
\def\TASK#1{}
\newdimen\margin  
\def\textno#1&#2\par{%
    \margin=\hsize
    \advance\margin by -4\parindent
           \setbox1=\hbox{\sl#1}%
    \ifdim\wd1 < \margin
       $$\box1\eqno#2$$%
    \else
       \bigbreak
       \hbox to \hsize{\indent$\vcenter{\advance\hsize by -3\parindent
       \sl\noindent#1}\hfil#2$}%
       \bigbreak
    \fi}

\global\long\def\E{\mathbb{E}}
\global\long\def\P{\mathbb{P}}
\global\long\def\var{\textnormal{var}}
\global\long\def\cov{\textnormal{cov}}


\global\long\def\C{\mathbb{C}}
\global\long\def\cC{\mathcal{C}}
\global\long\def\N{\mathbb{N}}
\global\long\def\cH{\mathcal{H}}
\global\long\def\cP{\mathcal{P}}
\global\long\def\Q{\mathbb{Q}}
\global\long\def\R{\mathbb{R}}
\global\long\def\Z{\mathbb{Z}}


\global\long\def\mod{\;(\textnormal{mod}\;}
\global\long\def\f{\mathbb{F}}


\global\long\def\norm#1{\left|\left|#1\right|\right|}
\global\long\def\ip#1#2{\langle#1,#2\rangle}
\global\long\def\per{\textnormal{per}}
\global\long\def\det{\textnormal{det}}
\global\long\def\tr{\textnormal{Tr}}


\global\long\def\join{\vee}
\global\long\def\meet{\wedge}


\global\long\def\eps{\varepsilon}
\global\long\def\supp{\textnormal{supp }}
\global\long\def\re{\textnormal{Re }}
\global\long\def\im{\textnormal{Im }}
\global\long\def\floor#1{\left\lfloor #1\right\rfloor }
\global\long\def\ceil#1{\left\lceil #1\right\rceil }

\def\noproof{{\unskip\nobreak\hfill\penalty50\hskip2em\hbox{}\nobreak\hfill%
       $\square$\parfillskip=0pt\finalhyphendemerits=0\par}\goodbreak}
\def\endproof{\noproof\bigskip}

\renewcommand{\labelenumi}{(\roman{enumi})}

\newrefformat{cor}{Corollary \ref{#1}}

\title{Optimal covers with Hamilton cycles in random graphs}
\author{Dan Hefetz, Daniela K\"uhn, John Lapinskas and Deryk Osthus}
\thanks {The research leading to these results was partially supported by the European Research Council
under the European Union's Seventh Framework Programme (FP/2007--2013) / ERC Grant
Agreement n. 258345 (D.~K\"uhn).}
\date{\today} 

\begin{abstract} 
A packing of a graph $G$ with Hamilton cycles is a set of edge-disjoint Hamilton cycles in $G$.
Such packings have been studied intensively and recent results imply that a largest packing of Hamilton cycles in $G_{n,p}$
a.a.s.~has size $\lfloor \delta(G_{n,p}) /2 \rfloor$.
Glebov, Krivelevich and Szab\'o recently initiated research on the `dual' problem, where
one asks for a set of Hamilton cycles covering all edges of $G$.
Our main result states that for $\frac{\log^{117}n}{n}\le p\le 1-n^{-1/8}$, a.a.s.~the edges of $G_{n,p}$ can be covered by 
$\ceil{\Delta(G_{n,p})/2}$ Hamilton cycles. This is clearly optimal and improves an approximate result 
of Glebov, Krivelevich and Szab\'o, which holds for $p \ge n^{-1+\eps}$.
Our proof is based on a result of Knox, K\"uhn and Osthus on packing Hamilton cycles in pseudorandom graphs.
\end{abstract}

\maketitle

\section{Introduction}
Given graphs $H$ and $G$, an $H$-decomposition of $G$ is a set of edge-disjoint copies of $H$ in $G$ which cover all edges of $G$.
The study of such decompositions forms an important area of Combinatorics but it is notoriously difficult.
Often an $H$-decomposition does not exist (or it may be out of reach of current methods).
In this case, the natural approach is to study the packing and covering versions of the problem.
Here an \emph{$H$-packing} is a set of edge-disjoint copies of $H$ in $G$ and an \emph{$H$-covering} 
is a set of (not necessarily edge-disjoint) copies of $H$ covering all the edges of $G$.
An $H$-packing is \emph{optimal} if it has the largest possible size and an $H$-covering is \emph{optimal} if it has the smallest possible size.
The two problems of finding (nearly) optimal packings and coverings may be viewed as `dual' to each other.

By far the most famous problem of this kind is the Erd\H{o}s-Hanani problem on packing and covering a complete $r$-uniform hypergraph with $k$-cliques,
which was solved by R\"odl~\cite{Rodlnibble}. In this case, it turns out that the (asymptotic) covering and packing versions of the problem are trivially equivalent
and the solutions have approximately the same value.

Packings of Hamilton cycles in random graphs $G_{n,p}$ were first studied by Bollob\'as and Frieze~\cite{BF85}.
(Here $G_{n,p}$ denotes the binomial random graph on $n$ vertices with edge probability $p$.)
Recently, the problem of finding optimal packings of edge-disjoint Hamilton cycles in a random graph has received a large amount 
of attention, leading to its complete solution in a series of papers by several authors (see below for more details on the history of the problem).
The size of a packing of Hamilton cycles in a graph $G$ is obviously at most $\lfloor \delta(G)/2 \rfloor$, and this trivial bound turns out to be tight in the case of $G_{n,p}$
for \emph{any} $p$. 

The covering version of the problem was first investigated by Glebov, Krivelevich and Szab\'o~\cite{GKS}.
Note that the trivial bound on the size an optimal covering of a graph $G$ with Hamilton cycles is $\lceil \Delta(G)/2 \rceil$.
They showed that for $p \ge n^{-1+\eps}$, this bound is a.a.s.~approximately tight, i.e.~in this range, 
a.a.s.~the edges of $G_{n,p}$ can be covered with $(1+o(1))\Delta(G_{n,p})/2$ Hamilton cycles.
Here we say that a property $A$ holds a.a.s.~(asymptotically almost surely), if the probability that $A$ holds tends to $1$ as $n$ tends to infinity.

The authors of~\cite{GKS} also conjectured that their approximate bound could be extended to any $p \gg \log n/n$.
We are able to go further and prove the corresponding exact bound, unless $p$ tends to $0$ or $1$ rather quickly.
\begin{thm}\label{thm:main-result}
Suppose that $G\sim G_{n,p}$, where $\frac{\log^{117}n}{n}\le p\le1-n^{-1/8}$.
Then a.a.s.~the edges of $G$ can be covered by $\ceil{\Delta(G)/2}$
Hamilton cycles.
\end{thm}
Note that the exact bound fails when $p$ is sufficiently large. Indeed, let $n\ge 5$ be odd and take $p = 1 - n^{-2}$. Then with $\Omega(1)$ probability\COMMENT{This happens with probability 
\[\binom{n}{2}\left(1-\frac{1}{n^{2}}\right)^{\binom{n}{2}-1}\frac{1}{n^2} \ge \frac{1}{4}\left(1-\frac{1}{n^2}\right)^{n^2} \ge \frac{1}{4}e^{-1-\frac{1}{n^2}} \ge \frac{1}{100}.\]}, $G\sim G_{n,p}$ is the complete graph with one edge $uv$ removed. We claim that in this case, $G$ cannot be covered by $(n-1)/2$ Hamilton cycles. Suppose such a cover exists. Then exactly one edge is contained in more than one Hamilton cycle in the cover. But $u$ and $v$ both have odd degrees, and hence are both incident to an edge contained in more than one Hamilton cycle. Since $uv \notin E(G)$, these edges must be distinct and we have a contradiction. 

Note also that even though our result does not hold for $p > 1 - n^{-1/8}$, it still implies the conjecture of~\cite{GKS} in this range. Indeed, if $G \sim G_{n,p}$ with $p > 1 - n^{-1/8}$, we may simply partition $G$ into two edge-disjoint graphs uniformly at random and apply Theorem~\ref{thm:main-result} to each one to a.a.s. cover $G$ with $(1+o(1))n/2$ Hamilton cycles.

Unlike the situation with the Erd\H{o}s-Hanani problem, the packing and covering problems are not equivalent in the case of Hamilton cycles.
However, they do turn out to be closely related, so we now summarize the known results leading to the solution of the packing problem for Hamilton cycles in 
random graphs. Here `exact' refers to a bound of $\lfloor \delta(G_{n,p})/2 \rfloor$, and $\varepsilon$ is a positive constant.
$$
\begin{array}{l|l|l}
\mbox{authors}   & \mbox{range of } p &   \\
\hline
\mbox{Ajtai, Koml\'os \& Szemer\'edi~\cite{Komlos}} &  \delta(G_{n,p}) =2 & \mbox{exact}  \\
\mbox{Bollob\'as \& Frieze~\cite{BF85}}   &  \delta(G_{n,p}) \mbox{ bounded} & \mbox{exact}  \\
\mbox{Frieze \& Krivelevich~\cite{fk}} & p \mbox{  constant} & \mbox{approx.}  \\
\mbox{Frieze \& Krivelevich~\cite{FK08}} & p=\frac{(1+o(1))\log n}{n} & \mbox{exact}  \\
\mbox{Knox, K\"uhn \& Osthus~\cite{AHDoRG}} & p \gg \frac{\log n}{n} & \mbox{approx.}  \\
\mbox{Ben-Shimon, Krivelevich \& Sudakov~\cite{BKS}} & \frac{(1+o(1))\log n}{n}\le p\le \frac{1.02\log n}{n} & \mbox{exact}  \\
\mbox{Knox, K\"uhn \& Osthus~\cite{Knox2011e}} & \frac{\log^{50}n}{n} \le p \le 1- n^{-1/5} & \mbox{exact}  \\
\mbox{Krivelevich \& Samotij~\cite{KrS}} & \frac{\log n}{n} \le p \le n^{-1 + \eps} & \mbox{exact}  \\
\mbox{K\"uhn \& Osthus~\cite{KOappl}} & p \ge 2/3 & \mbox{exact}  \\
\end{array}
$$
In particular, the results in~\cite{BF85,Knox2011e,KrS,KOappl} (of which~\cite{Knox2011e,KrS} cover the main range)
together show that for any $p$, a.a.s.~the size of an optimal packing of Hamilton cycles in $G_{n,p}$ is $\lfloor \delta(G_{n,p})/2 \rfloor$.
This confirms a conjecture of Frieze and Krivelevich~\cite{FK08} (a stronger conjecture was made in~\cite{fk}).

The result in~\cite{KOappl} is based on a recent result of K\"uhn and Osthus~\cite{monster} which guarantees the existence of a Hamilton decomposition in every 
regular `robustly expanding' digraph. The main application of the latter was the proof (for large tournaments)
of a conjecture of Kelly that every  regular tournament has a Hamilton decomposition.
But as discussed in~\cite{monster,KOappl}, the result in~\cite{monster} also has a number of further applications to packings of Hamilton cycles in dense graphs and (quasi-)random graphs.

Recall that the above results imply an optimal packing result for any $p$.
However, for the covering version, we need $p$ to be large enough to ensure the existence of at least one Hamilton cycle
before we can find any covering at all. This is the reason for the restriction $p \gg \log n/n$ in the conjecture of Glebov, Krivelevich and Szab\'o~\cite{GKS}
mentioned above.
However, they asked the intriguing question whether this might extend to $p$ which is closer to the threshold $\log n/n$
for the appearance of a Hamilton cycle in a random graph.
In fact, it would be interesting to know whether a `hitting time' result holds. For this, consider the well-known `evolutionary' random graph process $G_{n,t}$:
Let $G_{n,0}$ be the empty graph on $n$ vertices. Consider a random ordering of the edges of $K_n$. Let $G_{n,t}$
be obtained from $G_{n,t-1}$ by adding the $t$th edge in the ordering.
Given a property $\cP$, let $t(\cP)$ denote the \emph{hitting time} of $\cP$, i.e.~the smallest $t$ so that $G_{n,t}$ has $\cP$. 
\begin{question}
\label{con:hittime}
Let $\cC$ denote the property that an optimal covering of a graph $G$ with Hamilton cycles has size $\lceil \Delta(G)/2 \rceil$. 
Let $\cH$ denote the property that a graph $G$ has a Hamilton cycle.
Is it true that a.a.s.~$t(\cC)=t(\cH)$?
\end{question}
Note that $\cC$ is not monotone. In fact, it is not even the case that for all $t>t(\cC)$, $G_{n,t}$ a.a.s.~has $\cC$. Taking $n\ge 5$ odd and $t=\binom{n}{2}-1$, $G_{n,t}$ is the complete graph with one edge removed -- which, as noted above, may not be covered by $(n-1)/2$ Hamilton cycles. It would be interesting to determine (approximately) the ranges of $t$ such that a.a.s. $G_{n,t}$ has $\cC$.

The approximate covering result of Glebov, Krivelevich and Szab\'o~\cite{GKS} uses the approximate packing result in~\cite{AHDoRG} as a tool.
More precisely, their proof applies the result in~\cite{AHDoRG} to obtain an almost optimal packing.
Then the strategy is to add a comparatively small number of Hamilton cycles which cover the remaining edges.
Instead, our proof of Theorem~\ref{thm:main-result} is based on the main technical lemma (Lemma 47) of the exact packing result 
in~\cite{Knox2011e}. This is stated as Lemma~\ref{lem:prcovering} in the current paper and (roughly) states the following: Suppose we are given a regular graph $H$
which is close to being pseudorandom 
and a pseudorandom graph $G_1$, where $G_1$ is allowed to be surprisingly sparse compared to $H$.
Then we can find a  set of edge-disjoint Hamilton cycles in $G_1 \cup H$ covering all edges of $H$.
Our proof involves several successive applications of this result, where we eventually cover all edges of $G_{n,p}$.
In addition, our proof crucially relies on the fact that in the range of $p$ we consider,
there is a small but significant gap between the degree of the unique vertex $x_0$ of maximum degree and the other vertex degrees
(and the same holds for the vertex of minimum degree). This means that for all vertices $x  \neq x_0$, we can afford to 
cover a few edges incident to $x$ more than once. The analogous observation for the minimum degree was exploited in~\cite{Knox2011e} as well.

The result in~\cite{GKS} also holds for quasi-random graphs of edge density at least $n^{-1+\varepsilon}$, provided that they have an almost optimal packing of Hamilton cycles. It would be interesting to obtain such 
results for sparser quasi-random graphs too. 
In fact, the result in~\cite{Knox2011e} does apply in a quasi-random setting (see Theorem~48 in~\cite{Knox2011e}), but the assumptions are quite restrictive and it is not clear to which 
extent they can be used to prove results for $(n,d,\lambda)$-graphs, say. Note that even if the assumptions of~\cite{Knox2011e} could be weakened, our results would still not immediately generalise to $(n,d,\lambda)$-graphs.

This paper is organized as follows: In the next section, we collect several results and definitions regarding pseudorandom graphs, mainly from~\cite{Knox2011e}.
In Section~\ref{sec:tutte}, we apply Tutte's Theorem to give results which enable us to add a small number of
edges to certain almost-regular graphs in order to turn them into regular graphs (without increasing the maximum degree).
Finally, in Section~\ref{sec:proof} we put together all these tools to prove Theorem~\ref{thm:main-result}.

\section{Pseudorandom graphs}

The purpose of this section is to collect all the properties of $G_{n,p}$ that we need for our proof of Theorem~\ref{thm:main-result}.
Throughout the rest of the paper, we always assume that $n$ is sufficiently large for our estimates to hold. In particular, some
of our lemmas only hold for sufficiently large~$n$, but we do not state this explicitly. We write $\log$ for the natural logarithm and
$\log^a n$ for $(\log n)^a$. Given functions $f,g:\N\to \R$, we write $f=\omega(g)$ if $f/g\to \infty$ as $n\to \infty$.
We denote the average degree of a graph $G$ by $d(G)$.

We will need the following Chernoff bound (see e.g.~Theorem~2.1 in~\cite{JLR}).

\begin{lem}\label{lem:chernoff}
Suppose that $X\sim Bin(n,p)$.
For any $0<a<1$ we have $$\mathbb{P}(X \le (1-a)\mathbb{E}X) \le e^{-\frac{a^2}{3}\mathbb{E}X}.$$
\end{lem}

The following notion was first introduced by Thomason~\cite{T87}.
\begin{defn}\label{jumbled}
Let $p,\beta\ge 0$ with $p\le 1$. A graph $G$ is $(p,\beta)$-jumbled
if for all non-empty $S\subseteq V(G)$ we have $$\left|e_{G}(S)-p\binom{|S|}{2}\right|<\beta|S|.$$
\end{defn}
We will also use the following immediate consequence of Definition~\ref{jumbled}. Suppose that
$G$ is a $(p,\beta)$-jumbled graph and $X,Y\subseteq V(G)$ are disjoint. Then
\begin{equation}\label{eq:jumbled}
\left|e(X,Y)- p|X||Y|\right|\le 2\beta(|X|+|Y|).
\end{equation}
To see this, note that $e(X,Y)=e(X\cup Y)-e(X)-e(Y)$. Now~(\ref{eq:jumbled}) follows from Definition~\ref{jumbled}
by applying the triangle inequality.

The following notion was introduced in~\cite{Knox2011e}.%
    \COMMENT{This is not equivalent to requiring that if $G$ has degree sequence $d_{1}<d_{2}<\ldots<d_{n}$,
we have $d_{i}-d_{1}\ge 2(i-1)$ for all $i\le\log^{2}n+1$. We could e.g. have that $d_2=d_3=\delta+3$ and
$d_i\ge \delta+2\log^2 n$ for all $i\ge 4$.} 
\begin{defn}
Let $G$ be a graph on $n$ vertices. For a set $T\subseteq V(G)$,
let $\overline{d}_{G}(T):=\frac{1}{|T|}\sum_{t\in T} d_G(t)$ be the average degree of the vertices of
$T$ in $G$. Then $G$ is \emph{strongly $2$-jumping} if for
all non-empty $T\subseteq V(G)$  we have \[
\overline{d}_{G}(T)\ge\delta(G)+\min\{|T|-1,\log^{2}n\}.\]
\end{defn}
Note that a strongly $2$-jumping graph $G$ is `$2$-jumping', i.e.~it has a unique vertex of minimum degree and all other vertices
have degree at least $\delta(G)+2$.

The next definition collects (most of) the pseudorandomness properties that we need.

\begin{defn}\label{pseudodef}
A graph $G$ on $n$ vertices is \emph{$p$-pseudorandom} if all of
the following hold:
\begin{itemize}
\item[(P1)] $G$ is $(p,2\sqrt{np(1-p)})$-jumbled.
\item[(P2)] For any disjoint $S,T\subseteq V(G)$,
\begin{enumerate}
\item if $\left(\frac{1}{|S|}+\frac{1}{|T|}\right)\frac{\log n}{p}\ge\frac{7}{2}$,
then $e_{G}(S,T)\le2(|S|+|T|)\log n$,
\item if $\left(\frac{1}{|S|}+\frac{1}{|T|}\right)\frac{\log n}{p}\le\frac{7}{2}$,
then $e_{G}(S,T)\le7|S||T|p$.
\end{enumerate}
\item[(P3)] For any $S\subseteq V(G)$,
\begin{enumerate}
\item if $\frac{\log n}{|S|p}\ge\frac{7}{4}$, then $e(S)\le2|S|\log n$,
\item if $\frac{\log n}{|S|p}\le\frac{7}{4}$, then $e(S)\le\frac{7}{2}|S|^{2}p$.
\end{enumerate}
\item[(P4)] We have $np-2\sqrt{np\log n}\le\delta(G)\le np-200\sqrt{np(1-p)}$.
\item[(P5)] We have $\Delta(G)\le np+2\sqrt{np\log n}$.
\item[(P6)] $G$ is strongly 2-jumping.
\end{itemize}
\end{defn}
The following definition is essentially the same, except that some of the bounds are more restrictive.
\begin{defn}
A graph $G$ on $n$ vertices is \emph{strongly $p$-pseudorandom} if all of
the following hold:%
    \COMMENT{We could replace the lower bound on $\delta(G)$ in (SP4) by $np-\frac{15}{8}\sqrt{np\log n}$, but we don't use this stronger bound.}
\end{defn}
\begin{itemize}
\item[(SP1)] $G$ is $(p,\frac{3}{2}\sqrt{np(1-p)})$-jumbled.
\item[(SP2)] For any disjoint $S,T\subseteq V(G)$,
\begin{enumerate}
\item if $\left(\frac{1}{|S|}+\frac{1}{|T|}\right)\frac{\log n}{p}\ge\frac{7}{2}$,
then $e_{G}(S,T)\le \frac{3}{2}(|S|+|T|)\log n$,
\item if $\left(\frac{1}{|S|}+\frac{1}{|T|}\right)\frac{\log n}{p}\le\frac{7}{2}$,
then $e_{G}(S,T)\le6|S||T|p$.
\end{enumerate}
\item[(SP3)] For any $S\subseteq V(G)$,
\begin{enumerate}
\item if $\frac{\log n}{|S|p}\ge\frac{7}{4}$, then $e(S)\le \frac{3}{2}|S|\log n$,
\item if $\frac{\log n}{|S|p}\le\frac{7}{4}$, then $e(S)\le 3|S|^{2}p$.
\end{enumerate}
\item[(SP4)] We have $np-2\sqrt{np\log n}\le\delta(G)\le np-200\sqrt{np(1-p)}$.
\item[(SP5)] We have $\Delta(G)\le np + \frac{15}{8}\sqrt{np\log n}$.
\item[(SP6)] $G$ is strongly $2$-jumping.
\end{itemize}

The following lemma is an immediate consequence of Lemmas~9--11, 13 and~14 from \cite{Knox2011e}.%
    \COMMENT{The version on our homepage does not include the better bounds yet.}
\begin{lem}\label{lem:rgsarepr}
Let $G\sim G_{n,p}$, where $48^{2}\log^{7}n/n\le p\le 1-36\log^{\frac{7}{2}}n/\sqrt{n}$.
Then $G$ is strongly $p$-pseudorandom with probability at least $1-11/\log n$.
\end{lem}

The next observation shows that if we add a few edges at some vertex $x_0$ of a strongly pseudorandom graph such that
none of these edges is incident to the unique vertex of minimum degree, then we obtain a graph which is still pseudorandom.

\begin{lem}\label{lem:changingvertices}
Suppose that $G$ is a strongly $p$-pseudorandom graph with $p,1-p=\omega\left(1/n\right)$.
Let $y_1$ be the (unique) vertex of minimum degree in $G$ and let $x_0\neq y_1$ be any other vertex.
Let $F$ be a collection of edges of $K_n$ not contained in $G$ which are incident to $x_{0}$ but not to $y_1$
and such that $|F|\le \sqrt{np\log n}/8.$ Then the graph $G+F$ is $p$-pseudorandom.
\end{lem}

\begin{proof}
Let $G':=G+F$.
Clearly, (SP4) and (SP6) are not affected by adding the edges of $F$, so $G'$ satisfies (P4) and (P6).
The bound on $|F|$ together with (SP5) immediately imply that $G'$ satisfies (P5). %

We now show that $G'$ satisfies (P1). 
Indeed, for any $S\subseteq V(G')$, (SP1) implies that
\begin{eqnarray*}
\left|e_{G'}(S)-p\binom{|S|}{2}\right| & \le & \left|e_{G'}(S)-e_{G}(S)\right|+\left|e_{G}(S)-p\binom{|S|}{2}\right|\\
 & \le & |S|+\frac{3}{2}\sqrt{np(1-p)}|S|
  \le  2\sqrt{np(1-p)}|S|.
\end{eqnarray*}

To check (P2), suppose that $S,T\subseteq V(G')$ are disjoint. Without loss of generality we may assume that
$|S|\le|T|$. First suppose $\left(\frac{1}{|S|}+\frac{1}{|T|}\right)\frac{\log n}{p}\ge\frac{7}{2}$.
Then (i) of (SP2) implies that
\[e_{G'}(S,T)\le e_{G}(S,T)+|T|\le\frac{3}{2}\left(|S|+|T|\right)\log n+|T|\le2\left(|S|+|T|\right)\log n,\]
as required. Now suppose that $\left(\frac{1}{|S|}+\frac{1}{|T|}\right)\frac{\log n}{p}\le\frac{7}{2}$.
Then (ii) of (SP2) implies that%
     \COMMENT{Indeed, we have:
$\frac{1}{|S|}\cdot\frac{\log n}{p}\le\frac{7}{2}\Rightarrow p|S|\ge\frac{2}{7}\log n>1$ as desired.}
\[e_{G'}(S,T)\le e_{G}(S,T)+|T|\le|T|\left(6p|S|+1\right)\le 7|S||T|p.\]
So (ii) of (P2) holds.
The proof that (P3) holds is essentially the same.%
   \COMMENT{Now suppose $S\subseteq V(G)$ - we will show that $G'$ satisfies the
relevant upper bound on $e(S)$. First suppose $\frac{\log n}{|S|p}\ge\frac{7}{4}$
- then we have by (SP4) $e_{G'}(S)\le e_{G}(S)+|S|\le\frac{3}{2}|S|\log n+|S|\le2|S|\log n$
as desired. Now suppose $\frac{\log n}{|S|p}\le\frac{7}{4}$. Then
we have by (SP4) $e_{G'}(S)\le e_{G}(S)+|S|\le|S|\left(3p|S|+1\right)$
and so we must show $1\le\frac{p|S|}{2}$. But we have
$\frac{1}{|S|}\cdot\frac{\log n}{p}\le\frac{7}{4}\Rightarrow p|S|\ge\frac{4}{7}\log n>2$
as desired.}
\end{proof}

We say that a graph $G$ on $n$ vertices is \emph{$u$-downjumping} if it has
a unique vertex $x_{0}$ of maximum degree, and $d(x_{0})\ge d(x)+u$
for all $x\ne x_{0}$.
The following result follows from Lemma~17 in~\cite{Knox2011e} by considering complements.
The latter lemma in turn follows easily from Theorem~3.15 in~\cite{B84}.

\begin{lem}\label{lem:downjumping}
Let $G\sim G_{n,p}$ with $p,1-p=\omega\left(\log n/n\right)$.
Then a.a.s.~$G$ is $5\frac{\sqrt{np(1-p)}}{\log n}$-downjumping.
\end{lem}

The next result is intuitively obvious, but due to possible correlations between vertex degrees, it does merit some justification.

\begin{lem}\label{lem:maxmin}
Suppose that $\log^2 n/n < p' \le p  \le 1- \log^2 n/n$, that $p'\le 1/2$ and that $G\sim G_{n,p}$.
Let $H$ be a random subgraph of $G$ obtained by including each edge of $G$ into $H$ with probability $p'/p$. 
Then a.a.s.~$G$ contains a unique vertex $x_0$ of maximum degree and $x_0$ does not have minimum degree in $H$.
\end{lem}
\proof
Fix any $\varepsilon>0$.
Let $A$ be the event that $G$ contains a unique vertex $x_0$ of maximum degree and that $d_H(x_0)=\delta(H)$.
Let $f:=np'- \sqrt{np'\log \log n }$. Let $B$ be the event that $\delta(H) \le f$.
Note that $H \sim G_{n,p'}$. So Corollary 3.13 of~\cite{Bollobasbook} implies that $\P(\overline{B}) \le \eps$.
Let $C$ be the event that $G$ contains a unique vertex $x_0$ of maximum degree and that $d_H(x_0) \le f$ and note that $A \cap B \subseteq C$. 
Note also that $\P(A) \le \P(A \cap B )+ \P(\overline{B}) \le \P(C)+ \eps$. 
We say that a graph $F$ on $n$ vertices is \emph{typical} if $\Delta(F) \ge np$ and there is a unique vertex
of degree $\Delta(F)$. Now let $D$ be the event that $G$ is typical. Then Corollary 3.13 of~\cite{Bollobasbook} and
Lemma~\ref{lem:downjumping} together imply that
$\P(\overline{D}) \le \eps$. For any fixed graph $F$ on $n$ vertices, let $E_F$ denote the event that $G=F$.
Then $\P(C)\le \eps+ \sum_{F \colon F \ {\rm typical}} \P(C \mid E_F) \P(E_F)$.
Suppose that $E_F$ holds, where $F$ is typical. Let $N:=d_G(x_0)$ (note that $E_F$ determines $N$ and $x_0$).
Whether the event $C$ holds is now determined by a sequence of $N$ Bernoulli trials, each with success probability $p'/p$.
So let $X \sim \mbox{Bin}(N,p'/p)$. Then
$\E(X)=N(p'/p) \ge p'n$, which implies that%
    \COMMENT{This holds since the function $x-\sqrt{x\log\log n}$ is monotone increasing if $x\ge \log \log n/4$.}
$f \le \E(X)(1-\sqrt{\log \log n/\E(X)})$.   
Then an application of Lemma~\ref{lem:chernoff} gives us 
$$
\P(C \mid E_F)=\P(X \le f) \le e^{-\log \log n /3} \le \eps.
$$
So $\P(C) \le 2\eps$, which in turn implies that $\P(A) \le 3\eps$. Since $\eps$ was arbitrary, this implies the result.
\endproof 

Hefetz, Krivelevich and Szab\'o~\cite{HKS} proved a criterion for Hamiltonicity which requires only a rather weak quasirandomness notion.
We will use a special case of their Theorem~1.2 in~\cite{HKS}.
In that theorem, given a set $S$ of vertices in a graph $G$, we let $N(S)$ denote the external neighbourhood of $S$,
i.e.~the set of all those vertices $x\notin S$ for which there is some vertex $y\in S$ with $xy\in E(G)$.
Also, we say that $G$ is \emph{Hamilton-connected} if for any pair $x,y$ of distinct vertices
there is a Hamilton path with endpoints $x$ and $y$.
\begin{thm}
Suppose that $G$ is a graph on $n$ vertices which satisfies the following:
\begin{itemize}
\item[(HP1)] For every $S \subseteq V(G)$ with $|S| \le n/\sqrt{\log n}$, we have $|N(S)| \ge 20 |S|$.
\item[(HP2)] $G$ contains at least one edge between any two disjoint subsets $A,B \subseteq V(G)$ with $|A|,|B| \ge n/\log n$.%
\COMMENT{so this is the theorem with $d =20$ and where we bound the logs to make the statement cleaner.}
\end{itemize}
Then $G$ is Hamilton-connected.
\end{thm}

\begin{thm}\label{lem:hamilton-connected}
Let $G \sim G_{n,p}$ with $\log ^8 n /n \le p \le 1-n^{-1/3}$, and let $x_{0}$ be a vertex of maximum degree in $G$.
Then a.a.s.~$G-x_{0}$ is Hamilton-connected.
\end{thm}
\proof
It suffices to check that $G-{x_0}$ satisfies (HP1) and (HP2).
For $p$ in the above range, these properties are well known to hold a.a.s.~for $G$ with room to spare and so also hold for $G-{x_0}$. 
For completeness we point out explicit references.
To check (HP1), first note that Lemma~\ref{lem:rgsarepr} implies that $G$ is $p$-pseudorandom.
So Corollary~37 of~\cite{Knox2011e} applied with $A_x:=N_G(x) \setminus \{x_0\}$ now implies that (HP1) holds.
(HP2) is a special case of Theorem~2.11 in~\cite{Bollobasbook} -- the latter guarantees a.a.s.~the existence of many edges between
$A$ and $B$.%
     \COMMENT{When using (i) in Corollary 36, note that $\sum_{x \in S} |A_x|\ge s(np/2)\ge s\log^8 n/2$ and so
$|N(S)|=t\ge \sum_{x \in S} |A_x|/4\log n\ge s\log^6 n$.
When using~(ii), note that $\sum_{x \in S} |A_x|/7sp \ge s (np/2)/7sp =n/14$.
The results in HKS checking HP1 and HP2 do not seem to be formally applicable in our setting.}
\endproof


\section{Extending graphs into regular graphs}\label{sec:tutte}

The aim of this section is to show that whenever $H$ is a graph which satisfies certain conditions and $G$ is
a $p$-pseudorandom graph on the same vertex set which is edge-disjoint from $H$, then $G$ contains a spanning subgraph $H'$
whose degree sequence complements that of $H$, i.e.~such that $H\cup H'$ is $\Delta(H)$-regular.
The conditions on $H$ that we need are the following:
\begin{itemize}
\item $H$ has even maximum degree.
\item $H$ is $\sqrt{np}$-downjumping.
\item $H$ satisfies $\Delta(H)-\delta(H)\le (np\log n)^{5/7}$.
\end{itemize}
In order to show this we will use Tutte's $f$-factor theorem, for which we need to introduce the following notation.
Given a graph $G=(V,E)$ and a function $f:V\rightarrow\N\cup \{0\}$,
an \emph{$f$-factor} of $G$ is a subgraph $G'$ of $G$ such
that $d_{G'}(v)=f(v)$ for all $v\in V$. Our approach will then be to
set $f(v):=\Delta(H)-d_{H}(v)$ and attempt to find an $f$-factor
in the pseudorandom graph $G$. The following result of Tutte~\cite{Tutte1, Tutte2}
gives a necessary and sufficient condition for a graph to contain
an $f$-factor.

\begin{thm}\label{thm:Tutte}
A graph $G=(V,E)$ has an $f$-factor if and only
if for every two disjoint subsets $X,Y\subseteq V$, there are at most
\[\sum_{x\in X}f(x)+\sum_{y\in Y}(d(y)-f(y))-e(X,Y)\]
connected components $K$ of $G-X-Y$ such that
\[\sum_{x\in K}f(x)+e(K,Y)\]
is odd.
\end{thm}
When applying this result, we will often bound the number of components $K$ of $G-X-Y$ for which
$\sum_{x\in K}f(x)+e(K,Y)$ is odd by the total number of components of $G-X-Y$.
The next lemma (which is a special case of Lemma~20 in~\cite{Knox2011e}) implies that there are at most $|X|+|Y|$ such components.

\begin{lem}\label{lem:reghelper}
Let $G=(V,E)$ be a $p$-pseudorandom graph on $n$ vertices with $pn\ge \log n$. Then for any nonempty
$B\subseteq V$, the number of components of $G[V\setminus B]$ is at most $|B|$. In particular, $G$ is connected.
\end{lem}

The following lemma guarantees an $f$-factor in a pseudorandom graph, as long as $\sum_{v\in V}f(v)$ is even, $f(v)$ is not too large and for
all but at most one vertex $f(v)$ is not too small either. (Clearly, the requirement that $\sum_{v\in V}f(v)$ is even is necessary.)%
    \COMMENT{Before we had that $\sqrt{np}\le f(v)\le(np\log n)^{\frac{5}{7}}$ holds where the lower bound might fail for a bounded number $C$ of vertices.
But the case when $|X|<C$ and $Y=\emptyset$ was missing. So I changed the lemma since we are only using it in the case when $C=1$ anyway.}

\begin{lem}\label{lem:ffactors}
Let $G=(V,E)$ be a $p$-pseudorandom graph on $n$ vertices with%
    \COMMENT{Before we also had $p\le 1/2$, but this seems not to be needed for the proof.}
$pn\ge \log^{21}n$, and let $f:V\rightarrow\N\cup \{0\}$
be a function such that $\sum_{v\in V}f(v)$ is even. Suppose that $G$ contains a vertex $x_0$ such that
$f(x_0)$ is even and such that
$$ f(x_0)\le (np\log n)^{\frac{5}{7}} \ \ \ \text{and} \ \  \ \sqrt{np}\le f(v)\le(np\log n)^{\frac{5}{7}} \ \ \text{for all} \ \ v\in V\setminus\{x_0\}.$$
Then $G$ has an $f$-factor.
\end{lem}
\begin{proof}
Given two disjoint sets $X,Y\subseteq V$, we define $\alpha_{f}(X,Y)$
to be the number of connected components $K$ of $G-X-Y$ such that
\[\sum_{x\in K}f(x)+e(K,Y)\] is odd. We also define
\[\beta_{f}(X,Y):=\sum_{x\in X}f(x)+\sum_{y\in Y}\left(d(y)-f(y)\right)-e(X,Y).\]
By \prettyref{thm:Tutte}, it then suffices to prove that $\alpha_{f}(X,Y)\le\beta_{f}(X,Y)$. 

We will first show that $\alpha_{f}(X,Y)\le|X|+|Y|$. If either $X$ or $Y$
is nonempty, this follows immediately from \prettyref{lem:reghelper}.
If both $X$ and $Y$ are empty, then we must show that $\alpha_{f}(\emptyset,\emptyset)=0$. But
this holds since $G$ is connected by \prettyref{lem:reghelper},
and $\sum_{x\in V}f(x)$ is even by hypothesis. Hence $\alpha_{f}(X,Y)\le|X|+|Y|$
in all cases.

Hence if
\begin{equation}\label{eq:betaXY}
\beta_{f}(X,Y)\ge|X|+|Y|
\end{equation}
holds, then we have $\alpha_{f}(X,Y)\le\beta_{f}(X,Y)$ and we are done. If $X=Y=\emptyset$, (\ref{eq:betaXY}) holds. So it remains to consider the following cases.
\begin{caseenv}
\item $|X|= 1$.

\smallskip

Let $x$ denote the unique vertex in $X$. Suppose first that $Y=\emptyset$. 
In this case Lemma~\ref{lem:reghelper} implies that $G-x=G-X-Y$ is connected.
If $x=x_0$ then $\sum_{v\in V\setminus\{x\}} f(v)= \sum_{v\in V} f(v)-f(x)$ is even. Thus
$\alpha_{f}(X,Y)=0$ and so $\beta_{f}(X,Y)\ge \alpha_{f}(X,Y)$, as desired. If $x\neq x_0$
then $\beta_{f}(X,Y)=f(x)\ge \sqrt{np}\ge 1\ge \alpha_{f}(X,Y)$, as desired.

Thus we may assume that $Y\neq \emptyset$. Then
\begin{eqnarray*}
\beta_{f}(X,Y) & \ge  & \sum_{y\in Y}\left(d(y)-f(y)\right)-|X||Y|\\
& \stackrel{({\rm P4})}{\ge}  & \left(np-2\sqrt{np\log n}-(np\log n)^{\frac{5}{7}}\right)|Y|-|Y|\\
 & \ge  & \frac{np}{2}|Y|\ge |X|+|Y|
\end{eqnarray*}
and so~(\ref{eq:betaXY}) holds.
\item $|X|>1$ and $|Y|\le\frac{1}{4}|X|(np)^{-\frac{3}{14}}\log^{-\frac{5}{7}}n$.

\smallskip

Since $\sum_{y\in Y}d(y)\ge e(X,Y)$ it follows that in this case we have
\begin{align*}
\beta_{f}(X,Y) & \ge  \sum_{x\in X}f(x)-\sum_{y\in Y}f(y)
\ge (|X|-1)\sqrt{np}-|Y|(np\log n)^{\frac{5}{7}}\\
 & \ge  \frac{\sqrt{np}}{2}|X|-\frac{\sqrt{np}}{4}|X|\ge 2|X|\ge |X|+|Y|,
\end{align*}
and so~(\ref{eq:betaXY}) holds.
\item $1 < |X| \leq \frac{n}{2}$ and $|Y|>\frac{1}{4}|X|(np)^{-\frac{3}{14}}\log^{-\frac{5}{7}}n$.

\smallskip

It follows by (P1) and~(\ref{eq:jumbled}) that
\[e(X,Y)\le p|X||Y|+4\sqrt{np}(|X|+|Y|).\]
Thus
\begin{eqnarray}
\beta_{f}(X,Y)-\alpha_{f}(X,Y) & \ge & \sum_{y\in Y}\left(d(y)-f(y)\right)-e(X,Y)-|X|-|Y|\nonumber\\
 & \stackrel{({\rm P4})}{\ge} &  \left(np-2\sqrt{np\log n}-(np\log n)^{\frac{5}{7}}\right)|Y|-p|X||Y|-5\sqrt{np}(|X|+|Y|)\nonumber\\
 & \ge  & \left(p(n-|X|)-2(np\log n)^{\frac{5}{7}}\right)|Y|-5\sqrt{np}|X|\label{eq:beta2}\\
& \ge  & \left(\frac{np}{2}-2(np\log n)^{\frac{5}{7}}\right)|Y|-5\sqrt{np}|X|\nonumber\\
 & \ge  & \frac{1}{4}\left(\frac{(np)^{\frac{11}{14}}}{2\log^{\frac{5}{7}}n}-22\sqrt{np}\right)|X|\ge 0\nonumber,
\end{eqnarray}
as desired.%
   \COMMENT{For the last inequality we use that $(np)^{11/14}\ge 44(np)^{1/2}\log^{5/7}n$ since $np\gg \log^{5/2}n$.}
\item $|X|>\frac{n}{2}$ and $|Y|>\frac{1}{4}|X|(np)^{-\frac{3}{14}}\log^{-\frac{5}{7}}n$.

\smallskip

In this case we have
\[n-|X|\ge|Y|\ge\frac{|X|}{4(np)^{\frac{3}{14}}\log^{\frac{5}{7}}n}\ge\frac{n^{\frac{11}{14}}}{8p^{\frac{3}{14}}\log^{\frac{5}{7}}n}.\]
But as in the previous case, one can show that~(\ref{eq:beta2}) still holds and so
\begin{align*}
\beta_{f}(X,Y)-\alpha_{f}(X,Y) & \ge \left(p(n-|X|)-2(np\log n)^{\frac{5}{7}}\right)|Y|-5\sqrt{np}|X|\\
& \ge  \left(\frac{(np)^{\frac{11}{14}}}{8\log^{\frac{5}{7}}n}-2(np\log n)^{\frac{5}{7}}\right)|Y|-5\sqrt{np}|X|\\
 & \ge \frac{\left(np\right)^{\frac{11}{14}}}{9\log^{\frac{5}{7}}n}|Y|-5\sqrt{np}|X|\\
 & \ge  \left(\frac{(np)^{\frac{4}{7}}}{36\log^{\frac{10}{7}}n}-5\sqrt{np}\right)|X|\ge 0,\\
\end{align*}
as desired.%
   \COMMENT{For the third inequality we use that $(np)^{11/14}\ge 144 (np)^{5/7}\log^{10/7} n$ as $np\ge \log^{21}n$.
Similarly, the final inequality holds since $(np)^{4/7}\ge 180(np)^{1/2}\log^{10/7} n$, i.e. $(np)^{1/14}\ge 180\log^{10/7} n$
as $np\ge \log^{21}n$.}
\end{caseenv}
This completes the proof of the lemma.
\end{proof}

\begin{cor}\label{cor:exactregularise}
Let $G$ be a $p$-pseudorandom graph on $n$ vertices, where $pn\ge \log^{21}n$.
Suppose that $H$ is a graph on $V(G)$ which satisfies the following conditions:
\begin{itemize}
\item $H$ is $\sqrt{np}$-downjumping.
\item If $x_0$ is the unique vertex of maximum degree in $H$ then $H-x_0$ and $G-x_0$ are edge-disjoint.
\item $\Delta(H)$ is even.
\item $\Delta(H)-\delta(H)\le(np\log n)^{\frac{5}{7}}$.
\end{itemize}
Then there exists a $\Delta(H)$-regular graph $H'$ such that $H\subseteq H'\subseteq G\cup H$.
\end{cor}
\begin{proof}
Define $f(v):=\Delta(H)-d_{H}(v)$ for all $v\in V(G)$.
Then
\[\sum_{v\in V}f(v)=n\Delta(H)-\sum_{v\in V}d_{H}(v),\]
which is even. Moreover $f(x_{0})=0$ and our assumptions on~$H$ imply that
\[\sqrt{np}\le f(v)\le\Delta(H)-\delta(H)\le(np\log n)^{\frac{5}{7}}\]
for all $v\in V\setminus\{x_{0}\}$.
We may therefore apply \prettyref{lem:ffactors} to find an $f$-factor $G'$ in~$G$. Then $H':=H\cup G'$
is a $\Delta(H)$-regular graph as desired.
\end{proof}


\section{Proof of Theorem~\ref{thm:main-result}}\label{sec:proof}

The main tool for our proof of Theorem~\ref{thm:main-result} is the following result from \cite[Lemma~47]{Knox2011e}.
Roughly speaking, it asserts that given a regular graph $H_0$ which is contained in a pseudorandom graph $G$
and given a pseudorandom subgraph $G_0$ of $G$ which is allowed to be quite sparse compared to $H_0$, we can find a set of edge-disjoint Hamilton cycles in 
$H_0 \cup G_0$ which cover all edges of $H_0$. For technical reasons, instead of a single pseudorandom graph $G_0$, 
in its proof we actually need to  consider a union of several edge-disjoint pseudorandom graphs $G_1,\dots,G_{2m+1}$, where $m$ is close to $\log n$.

\begin{lem}\label{lem:prcovering}
Suppose that $p_{0}\ge\frac{\log^{14}n}{n}$ and
$p_{1}\ge\frac{(np_{0})^{\frac{3}{4}}\log^{\frac{5}{2}}n}{n}$. Let
$m:=\frac{\log(n^{2}p_{1})}{\log\log n}$, and for all $i\in[2m+1]$ set
$p_{i}:=p_{1}$ if $i$ is odd, and $p_{i}:=10^{10}p_{1}$ if $i$ is even.
Let $G$ be a $p_{0}$-pseudorandom graph on $n$ vertices. Suppose that
$G_{1},\ldots,G_{2m+1}$ are pairwise edge-disjoint spanning subgraphs of $G$
such that each $G_{i}$ is $p_{i}$-pseudorandom. Moreover, for all $i\in[2m+1]$,
let $H_{i}$ be an even-regular spanning subgraph of $G_{i}$ with
$\delta(G_{i})-1\le d(H_{i})\le\delta(G_{i})$. Suppose that $H_{0}$ is
an even-regular spanning subgraph of $G$ which is edge-disjoint from $\bigcup_{i=1}^{2m+1}H_{i}$.
Then there exists a collection $\mathcal{HC}$ of edge-disjoint Hamilton cycles such that the union $HC := \bigcup \mathcal{HC}$
of all these Hamilton cycles satisfies $H_{0}\subseteq HC\subseteq\bigcup_{i=0}^{2m+1}H_{i}$.
\end{lem}

The following lemma is a special case of Lemma 22(ii) of \cite{Knox2011e}. Given $p_i$-pseudo\-random
graphs $G_i$ as in Lemma~\ref{lem:prcovering}, it allows us to find the even-regular spanning subgraphs
$H_i$ required by Lemma~\ref{lem:prcovering}.

\begin{lem}
\label{lem:rfactors}Let $G$ be a $p$-pseudorandom graph on $n$
vertices such that $p,1-p=\omega\left(\log^{2}n/n\right)$.
Then $G$ has an even-regular spanning subgraph $H$ with $\delta(G)-1\le d(H)\le\delta(G)$.
\end{lem}

The next lemma ensures that $G\sim G_{n,p}$ contains a collection of Hamilton cycles which cover all edges of $G$ except for
some edges at the vertex $x_0$ of maximum degree and such that every edge at $x_0$ is covered at most once.
Theorem~\ref{thm:main-result} will then be an easy consequence of this lemma and Theorem~\ref{lem:hamilton-connected}.

\begin{lem}\label{thm:defactomainresult}
Let $G\sim G_{n,p}$, where $\frac{\log^{117}n}{n}\le p\le 1-n^{-\frac{1}{8}}$.
Then a.a.s.~$G$ has a unique vertex $x_0$ of degree $\Delta(G)$ and there
exist a collection $\mathcal{HC}$ of Hamilton cycles in $G$ and
a collection $F$ of edges incident to $x_{0}$ such that 
\begin{itemize}
\item[{\rm (i)}] every edge of $G-F$ is covered by some Hamilton cycle in $\mathcal{HC}$;
\item[{\rm (ii)}] no edge in $F$ is covered by a Hamilton cycle in $\mathcal{HC}$;
\item[{\rm (iii)}] no edge incident to $x_{0}$ is covered by more than one
Hamilton cycle in $\mathcal{HC}$.
\end{itemize}
\end{lem}
Note that in Lemma~\ref{thm:defactomainresult}, we have $|\mathcal{HC}| = (\Delta(G) - |F|)/2$.

The strategy of our proof of Lemma~\ref{thm:defactomainresult} is as follows.
We split $G\sim G_{n,p}$ into three edge-disjoint random graphs $G_1$, $G_2$ and $R$ such that the density of $G_1$ is
almost $p$ and both $G_2$ and $R$ are much sparser. 
It turns out we may assume that the vertex $x_0$ of maximum degree in $G$ also has maximum degree in $G_1$.
We then apply Corollary~\ref{cor:exactregularise} in order to extend $G_1$ into a $\Delta(G_1)$-regular
graph by using some edges of $R$. Next we apply Lemma~\ref{lem:prcovering} in order to cover this regular graph
with edge-disjoint Hamilton cycles, using some edges of~$G_2$.

Let $H_2$ be the subgraph of $R\cup G_2$ which is not
covered by these Hamilton cycles. Again, we can make sure that $x_0$ is still the vertex of maximum degree in~$H_2$.
We now apply Corollary~\ref{cor:exactregularise} again in order to extend $H_2$ into a $\Delta(H_2)$-regular
graph $H_2'$ by using edges of a random subgraph $R'$ of~$G_1$ (i.e.~edges which we have already covered by Hamilton cycles).
Finally, we would like to apply Lemma~\ref{lem:prcovering} in order to cover this regular graph by edge-disjoint
Hamilton cycles, using edges of another sparse random subgraph $G'$ of~$G_1$. However, this means that in the last step we might use edges of $G'$ at $x_0$,
i.e.~edges which have already been covered with edge-disjoint Hamilton cycles. Clearly, this would violate condition~(iii) of the lemma.

We overcome this problem as follows: at the beginning, we delete all those
edges at $x_0$ from $G_1$ which lie in $G'$, and then we regularize and cover the graph $H_1$ thus obtained from $G_1$ as before, instead of $G_1$ itself.
However, we have to ensure that $x_0$ is still the vertex of maximum degree in $H_1$. This forces us to make $G'$ quite sparse: the average degree of $G'$
needs to be significantly smaller than the gap between $d_G(x_0)=\Delta(G)$ and the degree of the next vertex, i.e.~significantly smaller than
$\sqrt{np(1-p)}/\log n$. Unfortunately it turns out that such a choice would make $G'$  too sparse to apply Lemma~\ref{lem:prcovering} in order to cover $H_2$. 
Thus the above two `iterations' are not sufficient to prove the lemma (where each iteration consists of an application of Corollary~\ref{cor:exactregularise} to regularize and
then an application of Lemma~\ref{lem:prcovering} to cover). But with three iterations, the above approach can be made to work.

\medskip

\noindent
\emph{Proof of Lemma~\ref{thm:defactomainresult}.}
Lemmas~\ref{lem:rgsarepr} and~\ref{lem:downjumping} imply that a.a.s.~$G$ satisfies the following two conditions:
\begin{itemize}
\item[(a)] $G$ is $p$-pseudorandom.
\item[(b)] $G$ is $5u$-downjumping, where $u:=\frac{\sqrt{np(1-p)}}{\log n}.$
\end{itemize}
Note that
\begin{equation}\label{eq:pnu}
(np)^{\frac{27}{64}}\log^{\frac{259}{32}}n=\frac{\sqrt{np(1-p)}}{\log n}\cdot\frac{\log^{\frac{291}{32}}n}{(np)^{\frac{5}{64}}\sqrt{1-p}}\le\frac{u}{2}.
\end{equation}
Indeed, to see the last inequality note that either $1-p\ge 1/2$ and $(np)^{\frac{5}{64}}\ge\log^{\frac{292}{32}}n$
or%
    \COMMENT{This is equivalent to $np \ge \log^{2\cdot 292/5}n=\log^{116.8} n$. So this is the point where we need that $np\ge \log^{117} n$.
For the next calculation we need the upper bound on~$p$.}
$(np)^{\frac{5}{64}}\ge (n/2)^{\frac{5}{64}}$
and $\sqrt{1-p}\ge n^{-\frac{1}{16}}$. So here we use the bounds on~$p$ in the lemma. Define
\begin{eqnarray*}
p_{2} & := & \frac{(np)^{\frac{3}{4}}\log^{\frac{7}{2}}n}{n}\ge\frac{\log^{91}n}{n},\\
p_{3} & := & \frac{(np_{2})^{\frac{3}{4}}\log^{\frac{7}{2}}n}{n}=\frac{(np)^{\frac{9}{16}}\log^{\frac{49}{8}}n}{n}\ge\frac{\log^{71}n}{n},\\
p'_{3} & := & 1600p_3,\\
p_{4} & := & \frac{(np_{3})^{\frac{3}{4}}\log^{\frac{7}{2}}n}{n}=\frac{(np)^{\frac{27}{64}}\log^{\frac{259}{32}}n}{n}\ge\frac{\log^{57}n}{n},\\
p_{1} & := & p-2p_{2}-p_{3},\\
m_{i} & := & \frac{\log(n^{2}p_{i})}{\log\log n}  \ \ \ \textnormal{for all} \ \ \ 2 \leq i \leq 4,\\
p_{(i,j)} & := & \begin{cases}
\frac{p_{i}}{(10^{10}+1)m_{i}+1} & \textnormal{ if } 2 \leq i \leq 4 \text{ and if } j \in [2m_i+1] \textnormal{ is odd},\\
\frac{10^{10}p_{i}}{(10^{10}+1)m_{i}+1} & \textnormal{ if } 2 \leq i \leq 4 \text{ and if } j \in [2m_i+1] \textnormal{ is even.}
\end{cases} 
\end{eqnarray*}
Now form random subgraphs of $G$ as follows. First partition $G$
into edge-disjoint random graphs $G_{1}$, $G_{2}$, $G_{3}$ and $R_{2}$
such that $G_{i}\sim G_{n,p_{i}}$ for $i = 1,2,3$ and $R_{2}\sim G_{n,p_{2}}$. (This can be done by randomly including each edge $e$ of $G$ into
precisely one of $G_{1}$, $G_{2}$, $G_{3}$ and $R_{2}$, where the probability that $e$ is included into $G_i$ is $p_i/p$ and
the probability that $e$ is included into $R_2$ is $p_2/p$, independently of all other edges of $G$.) We then choose
edge-disjoint random subgraphs $R'_{2}$, $R_{4}$ and $G_{4}$ of $G_{1}$
with $R'_{2}\sim G_{n,p_{2}}$, $R_{4}\sim G_{n,p_{4}}$, and $G_{4}\sim G_{n,p_{4}}$.
(Since $p_1\ge p_2+2p_4$ this can be done similarly to before.)
Next we choose a random subgraph $G'_3$ of $G_2$ such that $G'_{3}\sim G_{n,p'_{3}}$.
To summarize, we thus have the following containments, where $\dot{\cup}$ denotes the edge-disjoint union of graphs:
$$
G=G_1 \ \dot{\cup} \ G_2 \ \dot{\cup} \ G_3 \ \dot{\cup} \ R_2 \ \ \ \mbox{and} \ \ \
G_1 \supseteq R'_2 \ \dot{\cup}\ R_4 \ \dot{\cup}\ G_4 \ \ \ \mbox{and} \ \ \
G_2 \supseteq G'_3.
$$
Finally, for each $i\in\{2,3,4\}$, we partition $G_{i}$ into edge-disjoint random
subgraphs $G_{(i,1)},\ldots,G_{(i,2m_{i}+1)}$ with $G_{(i,j)}\sim G_{n,p_{(i,j)}}$.
Lemma~\ref{lem:rgsarepr} and a union bound implies that a.a.s.~the following conditions hold:
\begin{itemize}
\item[(c)] $G_{i}$ is $p_{i}$-pseudorandom for all $i=1,\dots,4$.
\item[(d)] $G_{(i,j)}$ is $p_{(i,j)}$-pseudorandom for all $i=2,3,4$ and all $j\in [2m_i+1]$.
\item[(e)] $R_{2}$ and $R'_{2}$ are $p_{2}$-pseudorandom, and $R_{4}$ is $p_{4}$-pseudorandom.
\item[(f)] $R_2\cup G_{2}\cup R'_{2}\cup G_{3}$
is strongly $(3p_{2}+p_{3})$-pseudorandom and $G'_3\cup G_{3}\cup R_{4}\cup G_{4}$
is strongly $(p'_3+p_{3}+2p_{4})$-pseudorandom.
\end{itemize}
Since $R_2\cup G_{2}\cup R'_{2}\cup G_{3}\sim G_{n,3p_{2}+p_{3}}$ and $G'_3\cup G_{3}\cup R_{4}\cup G_{4}\sim G_{n,p'_3+p_{3}+2p_{4}}$,
Lemma~\ref{lem:maxmin} implies that a.a.s.~the following condition holds:
\begin{itemize}
\item[(g)] Let $x_0$ be the unique vertex of maximum degree of $G$. Then $x_0$ is not the vertex of minimum degree
in $R_2\cup G_{2}\cup R'_{2}\cup G_{3}$ or $G'_3\cup G_{3}\cup R_{4}\cup G_{4}$.
\end{itemize}
It follows that a.a.s.~conditions (a)--(g) are all satisfied; in the remainder of the proof we will thus assume that they are. 
We can apply \prettyref{lem:rfactors}
for each $i=2,3,4$ and each $j\in [2m_i+1]$ to obtain an even-regular spanning subgraph $H_{(i,j)}$ of $G_{(i,j)}$
with $\delta(G_{(i,j)})-1\le d(H_{(i,j)})\le\delta(G_{(i,j)})$. 

As indicated earlier, our strategy consists of the following three iterations. The purpose of the first iteration is to cover all the edges of $G_1$. To do this, we will apply Corollary~\ref{cor:exactregularise} in order to extend $G_1$ into a regular graph $H'_1$, using some edges of $R_2$. (Actually we will first set aside a set $F_1$ of edges of $G_1$ at $x_0$, but this will still leave $x_0$ the vertex of maximum degree in $H_1:=G_1-F_1$. In particular, $F_1$ will contain the set $F^*$ of all edges of $G_4$ at $x_0$.) We will then apply Lemma~\ref{lem:prcovering} to cover $H'_1$ with edge-disjoint Hamilton cycles, using some edges of $G_2$.

The purpose of the second iteration is to cover all the edges of $G_2\cup R_2$ not already covered in the first iteration -- we denote this remainder by $H_2$. It turns out that $x_0$ will still be the vertex of maximum degree in $H_2$. If $\Delta(H_2)$ is odd, then we will add one edge from $F_1\setminus F^*$ to $H_2$ to obtain a graph $H'_2$ of even maximum degree. Otherwise, we simply let $H'_2:=H_2$. We extend $H'_2$ into a regular graph $H''_2$ using Corollary~\ref{cor:exactregularise} and some edges of $R'_2$, then cover $H''_2$ with edge-disjoint Hamilton cycles using Lemma~\ref{lem:prcovering} and some edges of $G_3$.

The purpose of the third iteration is to cover all the edges of $G_3$ not already covered in the second iteration -- we denote this remainder by $H_3$. We first add some (so far unused) edges from $F_1 \setminus F^*$ to $H_3$ in order to make $x_0$ the unique vertex of maximum degree. Let $H'_3$ denote the resulting graph. We then extend $H'_3$ into a regular graph $H''_3$ using Corollary~\ref{cor:exactregularise} and some edges of $R_4$, and finally cover $H''_3$ with edge-disjoint Hamilton cycles using Lemma~\ref{lem:prcovering} and some edges of~$G_4$. 

It is in this iteration that we make use of $G_3'$, for technical reasons. It turns out that $G_3 \cup G_4 \cup R_4$ is so sparse that adding the required edges from $F_1 \setminus F^*$ may destroy its pseudorandomness, rendering it unsuitable as a choice of $G$ in Lemma~\ref{lem:prcovering}. Since the only role of $G$ in Lemma~\ref{lem:prcovering} is that of a `container' for the other graphs, this issue is easy to solve by adding a slightly denser random graph to $G_3 \cup G_4 \cup R_4$, namely $G_3'$.

Note that we did not use any edges of $R'_2$ at $x_0$ when turning $H'_2$ into $H''_2$ since $x_0$ is a vertex of maximum degree in $H'_2$. Similarly, we did not use any edges of $R_4$ at $x_0$ when turning $H'_3$ into $H''_3$. Moreover, $F^*$ was the set of all edges of $G_4$ at $x_0$ and no edge in $F^*$ was covered in the first two iterations.
Altogether this means that we do not cover any edge at $x_0$ more than once.

Note that in the second and third iterations, the graphs $R'_2$ and $R_4$ we use for regularising consist of edges we have already covered. In the second iteration, this turns out to be a convenient way of controlling the difference between the maximum and minimum degree of $H_3$ (which might have been about $\Delta(G) - \delta(G)$ if we had used uncovered edges). In the third iteration, there are simply no more uncovered edges available.

After outlining our strategy, let us now return to the actual proof.
We claim that $x_{0}$ is the unique vertex of maximum degree in $G_{1}$ and that $G_{1}$ is $4u$-downjumping. Indeed, for all
$x\ne x_{0}$ we have
\begin{align*}
d_{G_{1}}(x) & = d_{G}(x)-d_{G_2\cup G_3\cup R_2}(x) \stackrel{({\rm b})}{\le} d_{G}(x_{0})-5u-d_{G_2\cup G_3\cup R_2}(x)\\
& =d_{G_1}(x_{0})+d_{G_2\cup G_3\cup R_2}(x_0)-5u-d_{G_2\cup G_3\cup R_2}(x)\\
& \le d_{G_{1}}(x_{0})+\Delta(G_{2})+\Delta(G_{3})+\Delta(R_{2})-5u-\delta(G_{2})-\delta(G_{3})-\delta(R_{2})\\
 & \le d_{G_{1}}(x_{0})-\left(5u-12\sqrt{np_{2}\log n}\right),
\end{align*}
where the last inequality follows from the facts that both $G_2$ and $R_2$ are $p_2$-pseudorandom,
$G_3$ is $p_3$-pseudorandom, $p_3\le p_2$ as well as from (P4) and~(P5).
But
\begin{equation}\label{eq:p2gap}
\sqrt{np_{2}\log n}=(np)^{\frac{3}{8}}\log^{\frac{9}{4}}n\stackrel{(\ref{eq:pnu})}{\le} \frac{u}{2} \cdot (np)^{-\frac{3}{64}}
\le \frac{u}{\log n}.
\end{equation}
Altogether this shows that $d_{G_1}(x)\le d_{G_1}(x_0)-4u$ for all $x\neq x_0$. Thus $G_1$ is $4u$-downjumping and $x_0$ is
the unique vertex of maximum degree in $G_{1}$, as desired. 
Note that
\begin{equation}\label{eq:DeltaG4}
\Delta(G_{4})\le 2np_{4}=2(np)^{\frac{27}{64}}\log^{\frac{259}{32}}n \stackrel{(\ref{eq:pnu})}{\le} u.
\end{equation}
Let $F^*$ be the set of all edges of $G_4$ which are incident to~$x_0$. Thus $|F^*|\le u$ by~(\ref{eq:DeltaG4}).
Choose a set $F_{1}$ of edges incident to $x_{0}$ in $G_{1}$ such that $F^*\subseteq F_1$,
\begin{equation}\label{eq:sizeF1}
3u-1\le|F_{1}|\le 3u,
\end{equation}
and such that $\Delta(G_{1}-F_{1})$ is even. Note that we used (\ref{eq:DeltaG4}) and thus the full strength of~(\ref{eq:pnu})
(in the sense that it would no longer hold if we replace 117 by 116 in the lower bound on $p$ stated in Lemma~\ref{thm:defactomainresult})
in order to be able to guarantee that $F^*\subseteq F_1$. So this is the point where we need the bounds on~$p$ in the lemma.
Let $H_{1}:=G_{1}-F_{1}$. Thus $H_{1}$ is still $u$-downjumping.

Our next aim is to apply \prettyref{cor:exactregularise} in order to extend $H_1$ into
a $\Delta(H_1)$-regular graph $H_1'$, using some of the edges of $R_{2}$. So we need to check that the conditions
in \prettyref{cor:exactregularise} are satisfied. But since $G_1$ is $p_1$-pseudorandom we have
\begin{align}
\Delta(H_{1})-\delta(H_{1}) & \le \Delta(G_{1})-\delta(G_{1}) \stackrel{({\rm P4}),({\rm P5})}{\le } 4\sqrt{n p_1\log n}\nonumber\\
 & \le 4\sqrt{n p\log n} = 4(np_{2})^{\frac{2}{3}}\log^{-\frac{11}{6}}n \le (np_{2}\log n)^{\frac{5}{7}}.\label{eq:pnp_2n}
\end{align}
Moreover $p_2 \ge \log^{21}n/n$ and $H_{1}$
is $u$-downjumping and so $\sqrt{np_{2}}$-downjumping by~(\ref{eq:p2gap}).
Since $R_2$ is $p_2$-pseudorandom we may therefore apply \prettyref{cor:exactregularise}
to find a regular graph $H_{1}'$ of degree $\Delta(H_1)$ with $H_{1}\subseteq H_{1}'\subseteq H_{1}\cup R_{2}$. 

Next, we wish to apply \prettyref{lem:prcovering} in order to cover $H_1'$ with edge-disjoint Hamilton cycles. Note that for every $1 \leq j \leq 2 m_2 + 1$
\begin{equation}\label{eq:p2j}
np_{(2,j)}\ge\frac{np_{2}}{(10^{10}+1)m_{2}+1}\ge\frac{(np)^{\frac{3}{4}}\log^{\frac{7}{2}}n\log\log n}{10^{11}\log n}\ge(np)^{\frac{3}{4}}\log^{\frac{5}{2}}n.
\end{equation}
So we can apply \prettyref{lem:prcovering} with $G$, $H'_1$, $G_{(2,1)},\dots,G_{(2,2m_2+1)}$ and $H_{(2,1)},\dots,H_{(2,2m_2+1)}$
playing the roles of $G$, $H_0$, $G_1,\dots,G_{2m+1}$ and $H_1,\dots,H_{2m+1}$ to obtain a collection $\mathcal{HC}_{1}$ of edge-disjoint
Hamilton cycles such that the union $HC_1 := \bigcup \mathcal{HC}_{1}$ of these Hamilton
cycles satisfies $$H_{1}'\subseteq HC_{1}\subseteq H_{1}'\cup\bigcup_{j=1}^{2m_2+1} H_{(2,j)}\subseteq H'_1\cup G_2.$$
Write $H_{2}:=(G_{2}\cup R_{2})\setminus E(HC_{1})$ for the uncovered
remainder of $G_{2}\cup R_{2}$. Note that 
\begin{itemize}
\item[(HC1)] no edge of $G$ incident to $x_{0}$ is covered more than once in $\mathcal{HC}_{1}$;
\item[(HC1$'$)] $HC_{1}$ contains no edges from $F_1$.
\end{itemize}

Our next aim is to extend $H_{2}$ into a regular graph $H'_2$ using some of the edges of $R'_{2}$. We will then use some of the edges
of $G_3$ in order to find edge-disjoint Hamilton cycles which cover $H'_2$.
Note that
\begin{equation}\label{eq:degreesH2}
d_{H_{2}}(x)=d_{H_{1}}(x)+d_{R_{2}\cup G_{2}}(x)-2|\mathcal{HC}_{1}|
\end{equation}
for all $x\in V(G)$.
Together with the fact that $H_1$ is $u$-downjumping this implies that for all $x\neq x_0$ we have
\begin{align*}
d_{H_{2}}(x_{0})-d_{H_{2}}(x) & = (d_{H_{1}}(x_{0})-d_{H_{1}}(x))+(d_{R_{2}\cup G_{2}}(x_{0})-d_{R_{2}\cup G_{2}}(x))\\
 & \ge u-(\Delta(R_{2})+\Delta(G_{2})-(\delta(R_{2})+\delta(G_{2})))\\
 & \ge u-8\sqrt{np_{2}\log n} \stackrel{(\ref{eq:p2gap})}{\ge} \sqrt{np_{2}}.
\end{align*}
(For the second inequality we used the fact that both $R_2$ and $G_2$ are $p_2$-pseudo\-random together with (P4) and~(P5).)
Thus $x_0$ is the unique vertex of maximum degree in $H_2$ and $H_2$ is $\sqrt{np_{2}}$-downjumping.
If $\Delta(H_2)$ is odd, let $H'_2$ be obtained from $H_2$ by adding some edge from $F_1\setminus F^*$.
Condition~(g) ensures that we can choose this edge in such a way that it is not incident to the unique vertex of minimum degree in the
$(3p_2+p_3)$-pseudorandom graph $R_2\cup G_2\cup R'_2\cup G_3$. Let $F'_1$ be the set consisting of this edge.
If $\Delta(H_2)$ is even, let $H'_2:=H_2$ and $F'_1:=\emptyset$.  
In both cases, let $F_2:=F_1\setminus F'_1$ and note that $H'_2$ is still $\sqrt{np_{2}}$-downjumping. 
Moreover,
\begin{eqnarray*}
\Delta(H_{2}')-\delta(H_{2}') & \le & \Delta(H_{2})-\delta(H_{2})+1\\
 & \stackrel{(\ref{eq:degreesH2})}{\le} & \Delta(H_{1})+\Delta(G_{2})+\Delta(R_{2})-\delta(H_{1})-\delta(G_{2})-\delta(R_{2})+1\\
 & \le & \Delta(G_{1})+\Delta(G_{2})+\Delta(R_{2})-\delta(G_{1})-\delta(G_{2})-\delta(R_{2})+1\\
 & \le & 4\sqrt{np_1\log n}+8\sqrt{np_2\log n}+1\le 5\sqrt{np\log n}\\
 & \le & (np_{2}\log n)^{\frac{5}{7}}.
\end{eqnarray*}
(For the fourth inequality we used the facts that $G_1$ is $p_1$-pseudorandom and both $R_2$ and $G_2$ are $p_2$-pseudorandom
together with (P4) and~(P5). The final inequality follows similarly to~(\ref{eq:pnp_2n}).)
Furthermore, note that $E(H'_2)\cap E(R'_2)\subseteq F'_1$ and so $H'_2-x_0$ and $R'_2-x_0$ are edge-disjoint.
Thus we may apply \prettyref{cor:exactregularise} to find a regular graph $H_{2}''$ of degree $\Delta(H_{2}')$
with $H_{2}'\subseteq H_{2}''\subseteq H_{2}'\cup R'_{2}$. Since
$x_{0}$ is of maximum degree in $H_{2}'$, we have the following:
\textno
No edge from $R'_{2}$ incident to $x_{0}$ was
added to $H'_2$ in order to obtain $H_{2}''$. &(\dagger)

Let $G^*_2:=(R_2\cup G_2\cup R'_2\cup G_3)+F'_1$. Our choice of $F'_1$ and condition~(f) together ensure that we can
apply \prettyref{lem:changingvertices} with $R_2\cup G_2\cup R'_2\cup G_3$ and $F'_1$ playing the roles of $G$ and $F$
to see that $G^*_2$ is $(3p_2+p_3)$-pseudorandom.
Note that for every $1 \leq j \leq 2 m_3 + 1$
$$
np_{(3,j)}\ge(4np_{2})^{\frac{3}{4}}\log^{\frac{5}{2}}n\ge(n(3p_{2}+p_{3}))^{\frac{3}{4}}\log^{\frac{5}{2}}n,
$$
where the first inequality follows similarly to~(\ref{eq:p2j}). 
Hence we may apply Lemma~\ref{lem:prcovering}
with $G^*_2$, $H''_2$, $G_{(3,1)},\dots,G_{(3,2m_3+1)}$ and $H_{(3,1)},\dots,H_{(3,2m_3+1)}$
playing the roles of $G$, $H_0$, $G_1,\dots,G_{2m+1}$ and $H_1,\dots,H_{2m+1}$ to obtain a collection $\mathcal{HC}_{2}$ of edge-disjoint
Hamilton cycles such that the union $HC_2 := \bigcup \mathcal{HC}_{2}$ of these Hamilton
cycles satisfies 
$$
H_{2}''\subseteq HC_{2}\subseteq H_{2}''\cup\bigcup_{j=1}^{2m_3+1} H_{(3,j)} \subseteq H_2'' \cup G_3.
$$
We now have the following properties:
\begin{itemize}
\item[(HC2)] no edge of $G$ incident to $x_{0}$ is covered more than once in $\mathcal{HC}_{1}\cup\mathcal{HC}_{2}$;
\item[(HC2$'$)] $HC_{1} \cup HC_2$ contains no edges from $F_2$;
\item[(HC2$''$)]  $\mathcal{HC}_{1}\cup\mathcal{HC}_{2}$ covers all  edges in $(G_1-F_2) \cup G_2 \cup R_2$.
\end{itemize}
Indeed, to see (HC2), first note that ($\dagger$) implies that all edges incident to $x_{0}$ in $HC_{2}$ are contained in $H_{2}'\cup G_{3}$
and thus in $(H_2+F'_1) \cup G_3$, which is edge-disjoint from $HC_{1}$.
Now (HC2) follows from (HC1) together with the fact that the Hamilton cycles in $\mathcal{HC}_{2}$ are pairwise edge-disjoint.

Write $H_{3}:=G_{3}\setminus E(HC_{2})$ for the subgraph of $G_3$ which is not covered by the Hamilton cycles in $\mathcal{HC}_{2}$.%
    \COMMENT{Recall that $R_2 \subseteq G_1$ and that all edges of $H_1 = G_1 -(F_1\setminus F'_1)$ 
have been covered by $\mathcal{HC}_{1}\cup\mathcal{HC}_{2}$.
So all the edges of $R'_2$ apart from those lying in $F_2$ are covered by $\mathcal{HC}_{1}\cup\mathcal{HC}_{2}$.
So unlike $R_2$ in the previous iteration, we do not need to consider $R_2'$ this time. Alternatively, note that (unlike $R_2$), $R_2'$ is not part of the partition of $G$, which
is why we don't need to consider it explicitly when covering $G$.}
Our final aim is to extend $H_{3}$ into a regular graph $H'_3$ using some of the edges of $R_{4}$.
We will then  use the edges of $G_4$ in order to find edge-disjoint Hamilton cycles which cover $H'_3$
(and thus the edges of $G_3$ not covered so far).
Note that for all $x\in V(G)$
\[d_{H_{3}}(x)=d(H_{2}'')+d_{G_{3}}(x)-2|\mathcal{HC}_{2}|.\]
Together with the fact that $G_3$ is $p_3$-pseudorandom this implies
that
\begin{equation}\label{eq:Delta3}
\Delta(H_{3})-\delta(H_{3})=\Delta(G_{3})-\delta(G_{3}) \stackrel{({\rm P4}),({\rm P5})}{\le } 4\sqrt{np_{3}\log n}.
\end{equation}
Thus we can add a set $F'_2\subseteq F_2\setminus F^*$ of edges at $x_0$ to $H_3$ to ensure that 
$x_0$ is the unique vertex of maximum degree in the graph $H'_3$ thus obtained from $H_3$,
that $H'_3$ is $\sqrt{np_{4}}$-downjumping, $\Delta(H_{3}')$ is even and such that
\begin{equation}\label{eq:sizeF'2}
|F'_2|\le 4\sqrt{np_{3}\log n}+\sqrt{np_4}+1\le 5\sqrt{np_{3}\log n}\le \sqrt{np_{2}\log n}\stackrel{(\ref{eq:p2gap})}{\le}\frac{u}{\log n}.
\end{equation}
Note that $|F_2\setminus F^*|=|F_1\setminus (F'_1\cup F^*)|\ge 2u-2$ by~({\ref{eq:sizeF1}) and since $|F^*|\le u$ by~(\ref{eq:DeltaG4}).
So we can indeed choose such a set $F'_2$.
Moreover, condition~(g) ensures that we can choose $F'_2$ in such a way that it contains no edge which is incident
to the unique vertex of minimum degree in the
$(p'_3+p_3+2p_4)$-pseudorandom graph $G'_3\cup G_3\cup R_4\cup G_4$. Let $F_3:=F_2\setminus F'_2$ and note that
\begin{align*}
\Delta(H_{3}')-\delta(H_{3}') & \le  \Delta(H_{3})-\delta(H_{3})+ \sqrt{np_4}+1 \stackrel{(\ref{eq:Delta3})}{\le} 5\sqrt{np_{3}\log n}
= 5(np_4)^{\frac{2}{3}}\log^{-\frac{11}{6}} n\\
&  \le (np_{4}\log n)^{\frac{5}{7}}.
\end{align*}
Furthermore, $E(H'_3)\cap E(R_4)\subseteq F'_2$ and so $H'_3-x_0$ and $R_4-x_0$ are edge-disjoint.
Since also $p_4 \ge \log^{21}n/n$, we may apply \prettyref{cor:exactregularise}
to obtain a regular graph $H_{3}''$ of degree $\Delta(H_{3}')$
such that $H_{3}'\subseteq H_{3}''\subseteq H_{3}'\cup R_{4}$. Note that since 
$x_{0}$ is of maximum degree in $H_{3}'$, we have the following:
\textno
No edge from $R_{4}$ incident to $x_{0}$ was added to $H'_3$ in order to obtain $H_{3}''$. 
& (\star)

Let $G^*_3:=(G'_3\cup G_3\cup R_4\cup G_4)+F'_2$. Since $|F'_2|\le 5\sqrt{np_{3}\log n}=\sqrt{np'_{3}\log n}/8$ by~(\ref{eq:sizeF'2}),
we may apply \prettyref{lem:changingvertices} with $G'_3\cup G_3\cup R_4\cup G_4$ and $F'_2$ playing the roles of $G$ and $F$
to see that $G^*_3$ is $(p'_3+p_3+2p_4)$-pseudorandom.
 
Note that for every $1 \leq j \leq 2 m_4 + 1$
$$
np_{(4,j)}\ge(4np'_{3})^{\frac{3}{4}}\log^{\frac{5}{2}}n\ge(n(p'_3+p_{3}+2p_{4}))^{\frac{3}{4}}\log^{\frac{5}{2}}n,
$$
where the first inequality follows similarly to~(\ref{eq:p2j}). Recall that $F^*$ denotes the set of all those edges of $G_4$
which are incident to $x_0$. Since $F'_2\cap F^*=\emptyset$,
$H''_3$ and $G_4$ are edge-disjoint (and so $H''_3, H_{(4,1)},\dots,H_{(4,2m_4+1)}$ are pairwise edge-disjoint).
Thus we can apply \prettyref{lem:prcovering}
with $G^*_3$, $H''_3$, $G_{(4,1)},\dots,G_{(4,2m_4+1)}$ and $H_{(4,1)},\dots,H_{(4,2m_4+1)}$
playing the roles of $G$, $H_0$, $G_1,\dots,G_{2m+1}$ and $H_1,\dots,H_{2m+1}$ to obtain a collection $\mathcal{HC}_{3}$ of edge-disjoint
Hamilton cycles such that the union $HC_3 := \bigcup \mathcal{HC}_{3}$ of these Hamilton
cycles satisfies 
$$
H_{3}''\subseteq HC_{3}\subseteq H_{3}''\cup\bigcup_{j=1}^{2m_4+1} H_{(4,j)} \subseteq H_3'' \cup G_4.
$$
We claim that no edge of $G$ incident to $x_{0}$ is covered more than once in
$\mathcal{HC}:=\mathcal{HC}_{1}\cup\mathcal{HC}_{2}\cup\mathcal{HC}_{3}$. Indeed, (HC2) implies that this was the case for
$\mathcal{HC}_{1}\cup\mathcal{HC}_{2}$. Moreover, recall that the Hamilton cycles in $\mathcal{HC}_{3}$ are pairwise edge-disjoint.
In addition, ($\star$) implies that all edges incident to $x_{0}$ in $HC_{3}$ are contained in 
$$
H'_{3}+F^*=H_3+F_2'+F^* \subseteq H_3 + F_2.
$$
So (HC2$'$) implies that none of these edges lies in $HC_{1}\cup HC_{2}$, which proves the claim. 

Note that (HC2$''$) and the definition of $\mathcal{HC}_{3}$ together imply that   $\mathcal{HC}$ covers all edges of $G-F_3$.
Let $F\subseteq F_3$ be the set of uncovered edges.
Then $F$ and $\mathcal{HC}$ are as required in the lemma.
\endproof

We remark that for the final application of \prettyref{lem:prcovering} in the proof of Lemma~\ref{thm:defactomainresult}
it would have been enough to consider $G_3\cup R_4\cup G_4$
instead of $G'_3\cup G_3\cup R_4\cup G_4$ (since $H''_3$ and all the $G_{(4,j)}$ are contained in $(G_3\cup R_4\cup G_4)+ F'_2$).
However, we would not have been able to apply \prettyref{lem:changingvertices} in this case since $|F'_2| > \sqrt{np_{3}\log n}/8$.
Introducing $G'_3$ ensures that the conditions of \prettyref{lem:changingvertices} are satisfied (and this is the only purpose of $G'_3$).

We can now combine Theorem~\ref{lem:hamilton-connected} and Lemma~\ref{thm:defactomainresult}
in order to prove Theorem~\ref{thm:main-result}.

\medskip

\noindent
\emph{Proof of Theorem~\ref{thm:main-result}.}
Lemma~\ref{thm:defactomainresult} implies that a.a.s.~$G$ contains
a collection $\mathcal{HC}$ of Hamilton cycles  and a collection $F$ of edges incident to the unique vertex $x_0$
of maximum degree such that no edge of $G$ incident to $x_{0}$ is contained in more than one Hamilton cycle in $\mathcal{HC}$
and such that the Hamilton cycles in $\mathcal{HC}$ cover precisely the edges of $G-F$.
Moreover, by Theorem~\ref{lem:hamilton-connected}, a.a.s.~$G-x_0$ is Hamilton-connected.

If $|F|$ is odd, we add one edge of $G-F$ incident to $x_0$ to~$F$. We still denote the resulting set of edges by~$F$.
Let $r:=|F|/2$ and $e_1e'_1,\dots,e_re'_r$ be pairs of edges such that $F$ is the union of all these $2r$ edges.
Since $G-x_0$ is Hamilton-connected, for each $1 \le i \le r$ there exists a
Hamilton cycle $C_i$ of $G$ containing both $e_i$ and $e'_i$. Then $\mathcal{HC}\cup \{C_1,\dots,C_r\}$ is a collection of
$\lceil \Delta(G)/2\rceil$ Hamilton cycles covering $G$, as desired.
\endproof

Using further iterations in the proof of Lemma~\ref{thm:defactomainresult}, one could reduce the exponent $117$ in Lemma~\ref{thm:defactomainresult}
(and thus in Theorem~\ref{thm:main-result}). One further iteration would lead to an exponent of $60$, while the effect of yet further iterations quickly becomes 
insignificant.%
\COMMENT{We set $p_5= (np)^{(3/4)^4} (\log n)^{\frac{259}{32} \cdot \frac{3}{4}+ \frac{7}{2}} \sim (np)^{1/2-0.18} (\log n)^{10.57-1}$,
and check that the analogue of~(\ref{eq:pnu}) still holds.
The upper bound can also be with a further iteration. With the methods of~\cite{KOappl}, it can even be improved to $1-O(1/n)$,
however the remaining very very dense case needs an extra argument, so all this does not seem worth mentioning or doing}

\medskip

{\footnotesize \obeylines \parindent=0pt

Dan Hefetz, John Lapinskas, Daniela K\"{u}hn, Deryk Osthus 
School of Mathematics 
University of Birmingham
Edgbaston
Birmingham
B15 2TT
UK
}
\begin{flushleft}
{\it{E-mail addresses}:\\
\tt{\{d.hefetz, jal129, d.kuhn, d.osthus\}@bham.ac.uk}}
\end{flushleft}

\end{document}